\documentclass{amsart}
\usepackage{amsfonts,amssymb,amscd,amsmath,enumerate,verbatim,calc,graphicx}
\newtheorem{theorem}{Theorem}[section]
\newtheorem{lemma}[theorem]{Lemma}
\newtheorem{proposition}[theorem]{Proposition}
\newtheorem{corollary}[theorem]{Corollary}
\theoremstyle{definition}
\theoremstyle{definitions}
\newtheorem{definition}[theorem]{Definition}

\newtheorem{remark}[theorem]{Remark}
\newtheorem{example}[theorem]{Example}
\theoremstyle{notations}

\theoremstyle{remarks}

\begin{document}

\author[T. Nasri and B. Mashayekhy]{Tayyebe Nasri$^{a,*}$ and Behrooz Mashayekhy$^b$}

\title[On Exact Sequences of the Rigid Fibrations]
{On Exact Sequences of the Rigid Fibrations}
\thanks{$^*$Corresponding author}
\thanks{E-mail addresses: t.nasri@ub.ac.ir, bmashf@um.ac.ir}
\maketitle

\begin{center}

{\it ${}^a$Department of Pure Mathematics, Faculty of Basic Sciences,  University of Bojnord,
 Bojnord, Iran.}\\
{\it ${}^b$Department of Pure Mathematics, Center of Excellence in Analysis on Algebraic Structures, Ferdowsi University of Mashhad,\\
P.O.Box 1159-91775, Mashhad, Iran.}\\

\end{center}


\begin{abstract}
In 2002, Biss investigated on a kind of fibration which is called rigid covering fibration (we rename it by rigid fibration) with properties similar to covering spaces. In this paper, we obtain a relation between arbitrary topological spaces and its rigid  fibrations. Using this relation we obtain a commutative diagram of homotopy groups and quasitopological homotopy groups and deduce some results in this field.
\end{abstract}
\section{Introduction and Motivation}
Endowed with the quotient topology induced by the natural surjective map $q:\Omega^n(X,x)\rightarrow \pi_n(X,x)$, where $\Omega^n(X,x)$ is the $n$th loop space of $(X,x)$ with the compact-open topology, the familiar homotopy group $\pi_n(X,x)$ becomes a quasitopological group which is called the quasitopological $n$th homotopy group of the pointed space $(X,x)$, denoted by $\pi_n^{qtop}(X,x)$ (see \cite{Bi,Br,Ga}).

In 2002, Biss \cite{Bi} investigated on a kind of fibration which is called rigid covering fibration (we rename it by rigid fibration) with properties similar to covering spaces. He proved that there is a universal rigid fibration for a space if and only if its topological fundamental group is totally path disconnected \cite[Theorem 4.3]{Bi}. The rigid fibrations generalize classical covering spaces, the hypothesis being that the topological fundamental group is totally disconnected rather than discrete as in the classical theory. Since rigid fibrations have unique path lifting, one can apply standard arguments to extend a number of classical theorems about covering spaces to the rigid
fibration setting. For instance, every morphism in the category of rigid fibrations over $X$ is itself a rigid fibration \cite[Lemma 4.5]{Bi}. Also, Biss proved that the coset space of any subgroup $\pi\leq \pi_1^{qtop}(HE)$ has no nonconstant paths, where $HE$ is the Hawaiian earring space. Therefore for any subgroup $\pi$, there is a rigid fibration $p :E\longrightarrow HE$ with $p_*\pi_1(E) = \pi$. On the other hands Brazas \cite{Bra} introduced the $\mathcal{C}$-covering map for a space $X$  which have a unique lifting property with
respect to maps on the objects of $\mathcal{C}$, where  $\mathcal{C}$ is the category of path-connected spaces having
the unit disk as an object and then he defined the category of $\mathcal{C}$-covering maps for a space $X$ denoted by $Cov_{\mathcal{C}}(X)$.  The $\mathcal{C}$-covering map generalize classical covering maps.  In this paper, we consider the category of rigid fibration maps of a space $X$, $RF(X)$, and show that it is a subcategory of $wCov_{\mathcal{C}}(X)$. Then we obtain some results in rigid fibrations. Also, we consider rigid fibrations of two topological spaces $X$ and $Y$ and obtain a relation between these spaces and its rigid fibrations. More precisely, if $p:E_H\longrightarrow X$ and $q:E_G\longrightarrow Y$ are two rigid fibrations of $X$ and $Y$,respectively  such that $p_*\pi_1(E_H)=H$ and $q_*\pi_1(E_G)=G$ and $g:X\longrightarrow Y$ is a continuous map such that $g_*(H)\subseteq G$, then there is a map $\tilde{f}:E_H\longrightarrow E_G$ such that $q\circ f= g\circ p$ and vise versa.
 Then with the above conditions, we obtain a commutative diagram in Set$_*$ with exact rows:
\[\begin{CD}
\cdots @> >> \pi_{n}(F_1)@> i_* >> \pi_n(E_H)@ > p_* >> \pi_n(X) @> d >> \pi_{n-1}(F_1) @> >> \cdots\\
&      & @VV (f|_{F_{1}})_{*} V@V f_*VV@V g_*VV@V (f|_{F_{1}})_{*}VV\\
\cdots @>  >> \pi_{n}(F_2) @ > j_* >>\pi_n(E_G)@>q_* >> \pi_n(Y)@> d' >> \pi_{n-1}(F_2) @> >> \cdots.
\end{CD}\]

In follow, using this diagram we deduce some results about  homotopy groups and quasitopological homotopy groups.
\section{Preliminaries}

In this section, we recall some of the main notion and results of rigid fibrations.
\begin{definition}\cite[Definition 2.1]{Bi}
Let $X$ be a topological space. A fibration $p :E\longrightarrow X$ is said to be a covering fibration if $p_* :\pi_i(E)\longrightarrow \pi_i(X)$ is an isomorphism for $i\geq 2$ and an injection for $i = 1$.
\end{definition}

Recall that we say that a fibration $p :E\longrightarrow X$ has unique path lifting property if for any path $\gamma :I\longrightarrow X$ in $X$ and $e\in E$  with $p(e) = \gamma (0)$, there is a unique lift $\tilde{\gamma} :I\longrightarrow E$ with $\gamma = p\circ\tilde{\gamma}$ and $\tilde{\gamma}(0) = e$. One can show that a fibration has unique path lifting property if and only if all of the fibers have no nonconstant paths \cite[Theorem 2.2.5]{Sp}.

\begin{definition}\cite[Definition 4.1]{Bi}\label{fib}
Let X be a topological space. A fibration $p :E\longrightarrow X$ is called a rigid
covering fibration if it is a covering fibration and if, in addition, each fiber has no
nonconstant paths.
\end{definition}

If we use the homotopy sequence of fibrations for rigid covering fibrations, we see that the condition "covering fibration" in Definition \ref{fib} can be replace with  "fibration". Therefore we can say that A fibration $p :E\longrightarrow X$ is called a rigid covering fibration (rename it by rigid fibration)  if each fiber has no
nonconstant paths.
\begin{theorem}\cite[Theorem 4.3]{Bi}\label{b1}
Let $X$ be a space and $\pi < \pi_1(X)$ a subgroup of the fundamental group
of $X$. If the left coset space $\pi_1^{qtop}(X)/\pi$ has no nonconstant paths, then there is a rigid
covering fibration $p :E\longrightarrow X$ with $p_*\pi_1(E) = \pi$.
\end{theorem}

For a pointed map $f:(X,x_0)\longrightarrow (Y,y_0)$, the mapping fiber is the pointed space $Mf=\{(x,\omega)\in X\times Y^I: \omega(0)=y_0 \ \ \text{and} \ \ \omega(1)=f(x) \}$, the base point of this space is $(x_0,\omega_0)$, where $\omega_0$ is the constant path at $y_0$. Also, there is an injection $k:\Omega (Y,y_0)\longrightarrow Mf$ given by $k(\omega)=(x_0, \omega)$ and an obvious map $\lambda:X\longrightarrow Mf$ by $\lambda(x)=(x,\omega_0)$ (see \cite{Ro}).
\section{Main Results}

In this section, we obtain some results in rigid fibrations and then we intend to obtain a relation between arbitrary topological spaces and its rigid fibrations. Then we obtain a commutative diagram of homotopy groups and quasitopological homotopy groups and deduce some results in this field.
\begin{theorem}
Let  $p :E\longrightarrow X$ be a rigid  fibration. If $A$ is any path component  of $E$, then $p|_A:A\longrightarrow p(A)$ is a rigid  fibration.
\end{theorem}
\begin{proof}
It follows from \cite[Lemma 2.3.1]{Sp} and Definition \ref{fib}.
\end{proof}
\begin{theorem}
Let  $p :E\longrightarrow X$ be a map. If $E$ is locally path connected, then $p$ is a rigid  fibration if and only if for each path component $A$ of $E$, $p(A)$ is a path component  of $X$ and  $p|_A:A\longrightarrow p(A)$ is a rigid  fibration.
\end{theorem}
\begin{proof}
It follows from \cite[Theorem 2.3.2]{Sp} and Definition \ref{fib}.
\end{proof}

The following theorems imply that if $\mathcal{C}$ is the category of connected locally path connected spaces, then every rigid fibration is a weak $\mathcal{C}$-covering map in the sense of \cite{Bra}.
\begin{theorem}\cite[Lemma 2.2.4]{Sp}\label{le1}
If a map has unique path lifting, it has the unique lifting property for path connected spaces.
\end{theorem}
\begin{theorem}\cite[Theorem 2.4.5]{Sp}
Let  $p :(\tilde{X}, \tilde{x_0})\longrightarrow (X, x_0)$ be a fibration with unique path lifting property. Let $Y$ be a connected locally path connected space. A necessary and sufficient condition that a map $f:(Y,y_0)\longrightarrow (X,x_0)$ have a lifting $(Y,y_0)\longrightarrow (\tilde{X}, \tilde{x_0})$ is that in $\pi_1(X,x_0)$
\[f_*\pi_1(Y,y_0)\subseteq p_*\pi(\tilde{X}, \tilde{x_0}).\]
\end{theorem}
We are inspired by the results of \cite{Bra} and we obtain the following results.

\begin{definition}
A subgroup $H\leq \pi_1(X,x_0)$ is a  rigid fibration subgroup (for simpricity, RF subgroup) if there is a  rigid fibration $p:(E, e_0)\longrightarrow (X,x_0)$  with $p_*\pi_1(E, e_0)=H$.
\end{definition}
\begin{remark}
Let $H$ be a subgroup of $\pi_1(X)$. If the left coset space $\pi_1^{qtop}(X)/H$ has no nonconstant paths, then $H$ is a RF subgroup of $\pi_1(X)$, by Theorem \ref{b1}. As an example every subgroup of fundamental group of Hawaiian earring, $\pi_1(HE)$, is a RF subgroup. Because   the  left coset space $\pi_1^{qtop}(HE)/H$ has no nonconstant paths, for every subgroup $H$ of $\pi_1(HE)$\cite{Bi}.
\end{remark}
\begin{example}\label{1}
Let $HA$ be the harmonic archipelago space. Any subgroup of $ \pi_1(HA)$ is not a RF subgroup. Indeed,  $\pi_1^{qtop}(HA)$ is indiscrete and therefore for any subgroup $H\leq \pi_1(HA)$, the  left coset space $\pi_1^{qtop}(HA)/H$ is not totally path disconnected.
\end{example}
\begin{theorem}
If $\{H_i\}_{i\in I}$ is any set of RF subgroup of  $\pi_1(X,x_0)$, then $\prod _i H_i$ is a rigid fibration subgroup.
\end{theorem}
\begin{proof}
If $H_i$ is a RF subgroup of  $\pi_1(X,x_0)$, then there is a  rigid fibration $p_i:(E_i, e_i)\longrightarrow (X,x_0)$  with $p_*\pi_1(E_i, e_i)=H_i$, for all $i\in I$. By putting $(E,e)=(\prod_i E_i, (e_i))$, the product $p=\prod p_i:(E,e)\longrightarrow (\prod_i X , x_0)$ is a rigid  fibration by \cite[Theorem 2.2.7]{Sp} and
\begin{align*}
p_*\pi_1(E, e) &= p_*\pi_1(\prod_i E_i, (e_i))\ \  \\
&= \prod_i  p_{i_*}\pi_1( E_i, e_i)\\
&=\prod_i  H_i.\\
\end{align*}
Thus $\prod _i H_i$ is a RF subgroup.
\end{proof}

For a given space $X$, let $RF(X)$ denote the category of rigid fibrations over
$X$ which is the category whose objects are rigid fibrations $p :E\longrightarrow X$ and morphisms
are commutative triangles of the form
\begin{center}
\includegraphics[height=2cm,width=3cm]{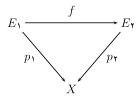}
\end{center}

 Let $DCov(X)$ denote the category of
disk-coverings over $X$  and $GSet$ denote the category of $G$-Sets
and $G$-equivariant functions. One can see that for a topological space $X$, $RF(X)\subseteq  DCov(X)$.
Recall that the functor $\mu: DCov(X)\longrightarrow \pi_1(X,x_0)Set$  was defined as follows: On objects, $\mu$ is defined as the fiber $\mu(p)=p^{-1}(x_0)$ . If $q : E_0 \longrightarrow X$ is a disk-covering and $f : E \longrightarrow E_0$ is a map such that $q\circ f = p$, then $\mu(f)$ is the restriction of f to a $\pi_1(X,x_0)$-equivariant function $p^{-1}(x_0)\longrightarrow q^{-1}(x_0)$. $\mu$ is a faithful functor \cite[Lemma 2.5]{Bra} and so the restriction on $RF(X)$ is also faithful. With the same argument of \cite[Theorem 2.11]{Bra}, we have the following theorem.
\begin{theorem}\label{2}
The functor $\mu:RF(X)\longrightarrow \pi_1(X,x_0)Set$ is fully faithful.
\end{theorem}

The functor $\mu$ in Theorem \ref{2} is not necessarily an equivalence of categories. Because subgroups $H\leq \pi_1(X,x_0)$ exist which are not RF subgroups(see Example \ref{1}).

Every morphism in $RF(X)$  is itself a rigid fibration \cite[lemma 4.5]{Bi}. By this fact and \cite[Theorem 2.2.6]{Sp} we have the following result.
\begin{proposition}\label{4}
Suppose  $p :E\longrightarrow X$ and  $q :X\longrightarrow Y$ are maps. Then \\
$(i)$  If $p$ and $q$ are rigid fibrations, then so is $q\circ p$,\\
$(ii)$ If $q$ and $q\circ p$ are  rigid fibrations, then so is $p$.
\end{proposition}

\begin{remark}
Let $g:X\longrightarrow Y$ be a rigid fibrations of $Y$, then there is a functor $\mathcal{G}:RF(X)\longrightarrow RF(Y)$ that is defines identity on morphisms and on objects as follows:
If  $p :E\longrightarrow X$ is a rigid fibration of $X$, we define $\mathcal{G}(p)=g\circ p$.
\end{remark}

The following theorem is one of the main results of this paper.
\begin{theorem}\label{13}
Let $p:E_H\longrightarrow X$ and $q:E_G\longrightarrow Y$ be two rigid fibrations of $X$ and $Y$,respectively  such that $p_*\pi_1(E_H)=H$ and $q_*\pi_1(E_G)=G$.\\
$(i)$ If $g:X\longrightarrow Y$ is a continuous map such that $g_*(H)\leq G$, then there is a map $f:E_H\longrightarrow E_G$ such that $q\circ f= g\circ p$.\\
$(ii)$  If $g:X\longrightarrow Y$ is a rigid fibration such that $g_*(H)\leq G$, then $f:E_H\longrightarrow E_G$ defined in part $(i)$ is a rigid fibration.\\
$(iii)$  If $g:X\longrightarrow Y$ is a continuous map and there is a map $f:E_H\longrightarrow E_G$ such that $q\circ f= g\circ p$, then $g_*(H)\leq G$.
\end{theorem}
\begin{proof}
$(i)$ Consider the map $g\circ p:E_H\longrightarrow Y$. Since $q$ is a rigid fibration, it has lifting property for path connected spaces. Since
\[(f\circ p)_*(\pi_1(H))=g_*(H)\leq G=q_*\pi_1(E_G),\]
 there exist a map $f:E_H\longrightarrow E_G$ such that $q\circ f= g\circ p$.\\
 $(ii)$ Since $g\circ p$ and $q$ are two rigid  fibration of $Y$, $f$ is a rigid  fibration of $E_H$ by Proposition \ref{4}.\\
 $(iii)$ Since $q\circ f= g\circ p$, by applying the functor $\pi_1$,
$q_*\circ f_*= g_*\circ p_*$ and then
\[
g_*(H)= g_*\circ p_*(\pi_1(E_H))=q_*\circ f_*(\pi_1(E_H))\leq q_*(\pi_1(E_G)) =G.\]
\end{proof}
\begin{theorem}
Let $H$, $G$ be two RF subgroups of  $\pi_1(X)$ and $\pi_1(Y)$, respectively. If $g:X\longrightarrow Y$ is a rigid fibration such that $g_*(H)\leq G$, then $g_*(H)$ is a RF subgroup of $\pi_1(Y)$.
\end{theorem}
\begin{proof}
By hypothises, there exist two rigid fibrations $p:E_H\longrightarrow X$ and $q:E_G\longrightarrow Y$ of $X$ and $Y$,respectively such that $p_*\pi_1(E_H)=H$ and $q_*\pi_1(E_G)=G$. Using Theorem \ref{13} part $(ii)$, there is a rigid fibration $f:E_H\longrightarrow E_G$ with $q\circ f= g\circ p$. The composition  $q\circ f:E_H\longrightarrow Y$ is a rigid fibration of $Y$ such that
\[ (q\circ f)_*(\pi_1(E_H))= (g\circ p)_*(\pi_1(E_H))=g_*(H).\]
\end{proof}

The following result can be concluded from part (i) of Theorem \ref{13}.
\begin{theorem}\label{ex12}
Let $p:E_H\longrightarrow X$ and $q:E_G\longrightarrow Y$ be two rigid fibrations of $X$ and $Y$ with fibers $F_1$ and $F_2$, respectively such that $p_*\pi_1(E_H)=H$ and $q_*\pi_1(E_G)=G$. Let $g:X\longrightarrow Y$ be a continuous map such that $g_*(H)\leq G$, then

$(i)$ There is a commutative diagram in Set$_*$ with exact rows:
\begin{equation}\label{ex1}
\begin{CD}
\cdots @> >> \pi_{n}(F_1)@> i_* >> \pi_n(E_H)@ > p_* >> \pi_n(X) @> d >> \pi_{n-1}(F_1) @> >> \cdots\\
&      & @VV (f|_{F_{1}})_{*} V@V f_*VV@V g_*VV@V (f|_{F_{1}})_{*}VV\\
\cdots @>  >> \pi_{n}(F_2) @ > j_* >>\pi_n(E_G)@>q_* >> \pi_n(Y)@> d' >> \pi_{n-1}(F_2) @> >> \cdots,
\end{CD}
\end{equation}
where the existence of the map $f:E_H\longrightarrow E_G$ follows from Theorem \ref{13}.

$(ii)$ There is a commutative diagram in Top$_*$ with exact rows:
\begin{equation}\label{ex2}
\begin{CD}
\cdots @> >> \pi_{n}^{qtop}(F_1)@> i_* >> \pi_n^{qtop}(E_H)@ > p_* >> \pi_n^{qtop}(X) @> d >> \pi_{n-1}^{qtop}(F_1) @> >> \cdots\\
&      & @VV (f|_{F_{1}})_{*} V@V f_*VV@V g_*VV@V (f|_{F_{1}})_{*}VV\\
\cdots @>  >> \pi_{n}^{qtop}(F_2) @ > j_* >>\pi_n^{qtop}(E_G)@>q_* >> \pi_n^{qtop}(Y)@> d' >> \pi_{n-1}^{qtop}(F_2) @> >> \cdots.
\end{CD}
\end{equation}
\end{theorem}
\begin{proof}
$(i)$ Since $g_*(H)\leq G$, there is a map  $f:E_H\longrightarrow E_G$ such that $q\circ f= g\circ p$, by Theorem \ref{13}. Since $f\circ i=j\circ f|_{F_1}$,  $q\circ f= g\circ p$ and $\pi_n$ is a functor, the first two square commute. To see commutativity of the last square, consider the following diagram:
\begin{equation}\label{di1}
\begin{CD}
\Omega X @ > k >> Mp @ <  \lambda << F_1\\
 @VV g_{\sharp} V@V l VV@V f|_{F_{1}} VV\\
\Omega Y @ > k' >> Mq @ <  \lambda' << F_2,
\end{CD}
\end{equation}
where the maps $g_{\sharp}$ and $l$ are induced maps. It is easy to see that  Diagram (\ref{di1}) is commutative. Therefore the induced diagram by the functor $\pi_n$ i.e. the last square in Diagram (\ref{ex2}), is commutative.\\
$(ii)$ It follows from a similar argument of part $(i)$ by applying the functor $\pi_n^{qtop}$.

\end{proof}
\begin{remark}
Let  $H < \pi_1(X)$, $G < \pi_1(Y)$ and the left coset spaces $\pi_1^{qtop}(X)/H$ and $\pi_1^{qtop}(Y)/G$ has no nonconstant paths. By Theorem \ref{b1}, Diagrams (\ref{ex1}) and (\ref{ex2}) are commutative in Set$_*$ and Top$_*$, respecively, where  $g:X\longrightarrow Y$ is a continuous map with $g_*(H)\leq G$.
\end{remark}

In follow, using this theorem we deduce some results about  homotopy groups and quasitopological homotopy groups.
\begin{lemma}\label{le}\cite[Lemma 6.2]{Ro}
Consider the commutative diagram with exact rows
\begin{equation*}
\begin{CD}
\cdots @> >> A_n@> i_n >> B_n @ > p_n >> C_n @> d_n >> A_{n-1} @> >> \cdots\\
&      & @VV f_n V@V g_n VV@V h_n VV@V f_{n-1} VV\\
\cdots @> >> A'_n@> j_n >> B'_n @ > q_n >> C'_n @> \Delta_n >> A'_{n-1} @> >> \cdots,
\end{CD}
\end{equation*}
in which every third vertical map $h_n$ is an isomorphism. Then there is an exact sequence
\begin{multline*}
\cdots \longrightarrow A_n\longrightarrow B_n\oplus A'_n\longrightarrow B'_n\longrightarrow A_{n-1} \longrightarrow \cdots.
\end{multline*}
\end{lemma}
\begin{corollary}\label{hom}
Let $p:E_H\longrightarrow X$ and $q:E_G\longrightarrow Y$ be two rigid fibrations of $X$ and $Y$ with fibers $F_1$ and $F_2$, respectively such that $p_*\pi_1(E_H)=H$ and $q_*\pi_1(E_G)=G$. If $X$ and $Y$ are homotopic, then there is a long exact sequence of homotopy groups.
\begin{multline}\label{seq}
\cdots \longrightarrow\pi_n(F_1)\longrightarrow \pi_n(E_H)\oplus \pi_n(F_2)\longrightarrow \pi_{n}(E_G)\longrightarrow \pi_{n-1}(F_1)\longrightarrow \cdots.
\end{multline}
In particular,

$(i)$ If $X$ and $Y$ are homotopic and $\pi_n(F_1)=0$, for all $n\geq 0$, then $\pi_{n}(E_G)\cong \pi_n(E_H)\oplus \pi_n(F_2)$, for all $n\geq 1$.

$(ii)$ If $X$ and $Y$ are homotopic and $E_G$ is contractible, then $\pi_n(F_1)\cong \pi_n(E_H)\oplus \pi_n(F_2)$, for all $n\geq 1$.

$(iii)$ If $X$ and $Y$ are homotopic, $E_H$ is contractible and $\pi_n(F_2)=0$, for all $n\geq 0$, then
$\pi_n(E_G)\cong \pi_{n-1}(F_1)$, for all $n\geq 1$.

\end{corollary}

\begin{proof}
Since $X$ and $Y$ are homotopic, each $g_*$ in Diagram (\ref{ex1}) is an isomorphism, so that Lemma \ref{le} gives the result.
\end{proof}

As an example of the part $(i)$ in above corollary, if $p:E_H\longrightarrow X$ is a covering space, $\pi_n(F_1)=0$, for all $n\geq 1$, then $\pi_{n}(E_G)\cong \pi_n(E_H)\oplus \pi_n(F_2)$, for all $n\geq 2$.
\begin{corollary}\label{five}
Let $p:E_H\longrightarrow X$ and $q:E_G\longrightarrow Y$ be two rigid fibrations of $X$ and $Y$ with fibers $F_1$ and $F_2$, respectively such that $p_*\pi_1(E_H)=H$ and $q_*\pi_1(E_G)=G$. If $X\simeq Y$ and $F_1\simeq F_2$, then for all $n\geq 1$
$$\pi_n(E_G)\cong \pi_n(E_H).$$
\end{corollary}
\begin{proof}
It follows from Five Lemma and Diagram (\ref{ex1}).
\end{proof}
Let $X$ be a space and $\pi < \pi_1(X)$ a subgroup of the fundamental group
of $X$. If the left coset space $\pi_1^{qtop}(X)/\pi$ has no nonconstant paths, then there is a rigid
 fibration $p :E\longrightarrow X$ with $p_*\pi_1(E) = \pi$.
\begin{remark}
$(i)$ Applying Theorem \ref{b1}, the results of Corollaries \ref{hom} and \ref{five} hold, provided the left coset spaces $\pi_1^{qtop}(X)/H$ and $\pi_1^{qtop}(Y)/G$ has no nonconstant paths. \\
$(ii)$ Using part $(ii)$ of Theorem \ref{ex12} and a similar argument we can conclude Corollaries \ref{hom} and \ref{five} for $\pi_n^{qtop}$.
\end{remark}







\bibliography{mybibfile}

\begin{thebibliography}{9999}

\bibitem{Bi}  D. Biss,  The topological fundamental group and generalized covering spaces, \textit{ Topology and
its Applications}, {\bf 124} (2002)  355--371.

 \bibitem{Br}  J. Brazas, The topological fundamental group and free topological groups, \textit{ Topology and its Applications} {\bf 158}(2011), 779--802.

\bibitem{Bra}   J. Brazas, Generalized covering space theories,  \textit{ Theory Appl. Categ.}
 {\bf 30} (2015) 1132--1162.
 \bibitem{Ga}  H. Ghane, Z. Hamed, B. Mashayekhy and H. Mirebrahimi,
Topological homotopy groups, \textit{ Bull. Belg. Math. Soc. Simon Stevin} {\bf 15:3}(2008), 455--464.

\bibitem{Ro} Joseph J. Rotman,
             \textit{An Introduction to Algebraic Topology},
             Springer-Verlag New York, 1988.

\bibitem{Sp} Edwin H. Spanier , \textit{Algebraic Topology},  Springer-Verlag New York, 1960.




\end{thebibliography}

\end{document}